\documentclass[a4paper,11pt]{article}
\usepackage[numbers]{natbib}
\usepackage{cleveref}
\usepackage{amsmath}
\usepackage{amsfonts}
\usepackage{amssymb}
\usepackage{mathrsfs}
\usepackage{amsthm}
\usepackage{graphicx}
\usepackage{pdfpages}
\newtheorem{defn}{Definition}[section]
\newtheorem{ex}{Example}[section]
\newtheorem{n}{Note}[section]
\newtheorem{theorem}{Theorem}[section]
\newtheorem{co}{Corollary}[section]
\newtheorem{lem}{Lemma}[section]
\newtheorem{p}{Proposition}[section]

\newtheorem{cun}{Cunjecture}[section]
\begin{document}
	\title{Introducing a new concept of distance on a topological space by generalizing the definition of quasi-pseudo-metric
	}
	\author{Hamid Shobeiri\footnote{Email address: shobeiri@mail.kntu.ac.ir (A.H. Shobeiri)}\\
		\footnotesize{{\it Department of Mathematics, K.N. Toosi
				University of Technology,}} \\
		\footnotesize{{\it P.O.Box 16315-1618, Tehran, Iran}}\\
	}
	\date{}
	\maketitle
	\begin{abstract}
		In this paper, a new structure is defined on a topological space that equips the space with a concept of distance in order to do that firstly, a generalization of quasi-pseudo-metric space named R.O-metric space is introduced, and some of its basic properties is studied. Afterwards the concept of generalized R.O-metric space is defined .Finally, we establish that every topological space is generalized R.O-metrizable.
	\end{abstract}
	\section{Introduction}
	Topological spaces are extension of metric spaces. It is well known that each arbitrary topological space is not necessary metrizable (see \cite{munkres1975topology} or \cite{simmons1963introduction}). Therefore despite of the beauty and simplicity of such extension, it involves some limitations. For example, size of neighborhoods of two distinct points are not comparable in topological spaces. In addition, uniform continuity, Cauchy sequence and complete space are no more definable in arbitrary topological spaces. These limitations may raise the idea of defining topological spaces through a new concept of distance, in order to simplify working in these spaces.Defining a new concept of distance, will be useful. In this direction, some mathematicians introduced some structures weaker than metric spaces.
	
	A metric on a set $X$ is a function $d:X\times X\to [0,\infty)$ such that for all $x,y,z\in X$, the following conditions are satisfied:
	\begin{align}
	d(x,y)=0\Leftrightarrow x=y,\label{I.1}\\ 
	d(x,y)=d(y,x)\label{I.2},\\
	d(x,z)\leq d(x,y)+d(y,z). \label{I.3}
	\end{align}
	One of the generalized metric spaces is semi-metric space that is introduced by Frechet and Menger which satisfies conditions \eqref{I.1} and \eqref{I.2} of definition of metric space (see \cite{cicchese1976questioni}, \cite{wilson1931semi}, \cite{frechet1906quelques} and \cite{menger1928untersuchungen}). In the last few years, the study of non-symmetric topology has received a new derive as a consequence of it's applications to the study of several problems in theoretical computer science and applied physics.
	One of such structures is quasi-metric space that is introduced by W.A.Wilson (see \cite{wilson1931quasi}) which has conditions \eqref{I.1} and \eqref{I.3}. One other generalization of metric spaces is called pseudo-metric space, which satisfies conditions $(\displaystyle d(x,x)=0)$, (2) , (3) (see \cite{simmons1963introduction}). Quasi-pseudo-metric space is introduced by Kelly (see \cite{kelly1963bitopological}) which satisfies conditions $( d(x,x)=0)$ and (3). $T_0$-quasi-metric space is quasi-pseudo-metric space that satisfies condition $(d(x,y)=0=d(y,x)\Rightarrow x=y)$ that is presented in paper \cite{kemajou2012isbell}.
	Multi-metric space is defined by Smarandache $($see \cite{smarandache2000mixed}, \cite{smarandache2001unifying}$)$, which is a union $\tilde{M} =\bigcup_{i=1}^{n}M_i$, such that each $M_i$ is a space
	with metric $d_i$ for all $1\le i\le n$.
	
	The above mentioned structures can not describe all topological spaces. In this paper, it is aimed to present a new structure to be able to describe all topological spaces. It is started by definition of structure that is called R.O-metric space (Right-Oriented-metric space: this terminology comes from non-symmetric meter)  which is a generalization of quasi-pseudo-metric space. Then generalized R.O-metric is defined which reforms the definition of topological space. In the first section, the concept of R.O-metric space is defined. In the second section, R.O-metric space is generalized and improved by adding some conditions.
	\section{Preliminaries}
	\begin{defn}
		Let $X$ be a non empty set; a function  $\overrightarrow{d}:X\times X\to [0,\infty)$ is called a R.O-metric on $X$ iff for every $x,y,z\in X$, the following conditions hold:
		\begin{enumerate}
			\item$\overrightarrow{d}(x,x)=0$,
			\item $\overrightarrow{d}(x,z)+\overrightarrow{d}(z,y)\neq 0 \Rightarrow \overrightarrow{d}(x,y)\le \overrightarrow{d}(x,z)+\overrightarrow{d}(z,y),$
		\end{enumerate}
		and then $(X,\overrightarrow{d})$ is called a R.O-metric space.
	\end{defn}
	\begin{ex}
		Every metric space is a R.O-metric space.
	\end{ex}
	\begin{ex}
		As another example , for $X=\{a,b,c\}$ consider $$\overrightarrow{d}(a,b)=\overrightarrow{d}(b,c)=1\; , \overrightarrow{d}(a,c)=2\; ,$$$$\overrightarrow{d}(a,a)=\overrightarrow{d}(b,b)=\overrightarrow{d}(c,c)=\overrightarrow{d}(b,a)=\overrightarrow{d}(c,a)=\overrightarrow{d}(c,b)=0,$$ then $(X,\overrightarrow{d})$ is a R.O-metric space.\ref{fig:55}

\begin{figure}
	\centering
	\includegraphics[width=0.7\linewidth]{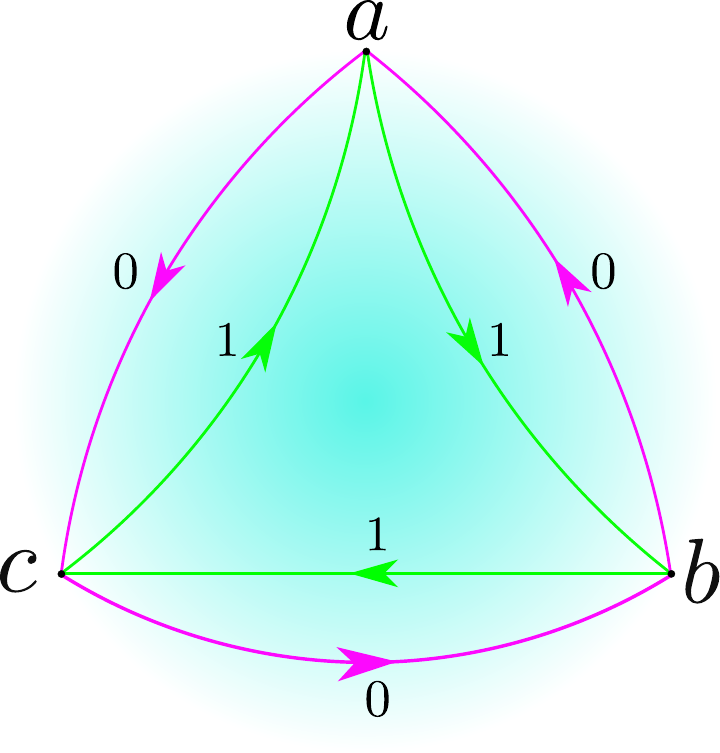}
	\label{fig:55}
\end{figure}

	\end{ex}
	\begin{defn}
		In a R.O-metric space $(X,\overrightarrow{d})$,
		the set, $V_{r}(p)=\{x\;|\;\overrightarrow{d}(p,x)<r\}$ is called a $r$-ball of a point $p$ with radius $r>0$.
	\end{defn}
	\begin{n}
		Let $(X,\overrightarrow{d})$ be a R.O-metric space ,
		then the set $S_{\overrightarrow{d}}=\{V_{r}(p)\;|\;r>0,p\in X\}$ is a subbasis for a topology on X, which is called the generated topology by $\overrightarrow{d}$ and is shown by $\tau_{\overrightarrow{d}}.$
	\end{n}
	\begin{defn}
		Topological space $(X,\tau)$ is called R.O-metrizable iff there exists a R.O-metric $\overrightarrow{d}$ such that $\tau_{\overrightarrow{d}}=\tau$.
	\end{defn}
	
	It can be shown that many of the most familiar topological spaces are R.O-metric spaces, here are some examples of non metrizable topological spaces which are R.O-metric spaces:
	\begin{ex}
		Let $X$ be a set and $\phi \neq A\subseteq X$, then $\tau_{A}=\{B\subseteq X \;|\; A\subseteq B\}\cup \{\emptyset\}$ is a topology on $X$; we define R.O-metric $\overrightarrow{d}$ on $X$ as follows:
		\begin{enumerate}
			\item $\forall x\in X;\;\overrightarrow{d}(x,x)=0,$  \item $\forall a \in X , \forall b\in A, \; \overrightarrow{d}(a,b)=0,$\item $\forall a\in X , \forall b\in A^{c} ;\; \overrightarrow{d}(a,b)=1.$
		\end{enumerate}
		
		The topology $\tau_{\overrightarrow{d}}$ is generated by subbasis $$S_{\overrightarrow{d}}=\{V_{r}(p)\;|\;r>0,p\in X\}=\{V_{1}(a)\;|\;a\in A\}\cup{\{V_{1}(b)\;|\;b\in A^{c}\}}\cup{\{V_{2}(x)\;|\;x\in X\}}$$
		$$\;\;\;\;\;\;\;\;\;\;\;\;\;\;\;\;\;\;\;\;\;\;\;\;\;=\{A\}\cup{\{\{b\}\cup{A}\;|\;b\in A^{c}\}}\cup{\{X\}}.$$
		Thus $\tau_{\overrightarrow{d}}=\tau_{A}$.
	\end{ex}
	\begin{ex}\label{13}
		Let $\tau$ be cofinite topology on infinite set $X$, that means topology in which the open sets are the subset of $X$ with finite complements. It is known that (see \cite{hrbacek1999introduction}) the set $X$ can be written as $X=\bigcup_{\alpha\in I}{A_{\alpha}}$ such that for all $\alpha\in I$, $A_{\alpha}$ is countable and $A_{\alpha}=\{x_{\alpha,1},x_{\alpha,2},...\}$. Now define R.O-metric $\overrightarrow{d}$ on $X$ as follows:
		\begin{enumerate}
			\item$\forall x\in X;\;\overrightarrow{d}(x,x):=0,$\item$\forall\alpha\in I,\;\forall n\in O;(O\;is\;the\;odd\;natural\;numbers)\\\;\overrightarrow{d}(x_{\alpha,n},x_{\alpha,n+1}):=1=\overrightarrow{d}(x_{\alpha,n+1},x_{\alpha,n})$,\item$\overrightarrow{d}(x_{\alpha,n},y):=0=\overrightarrow{d}(x_{\alpha,n+1},z),\;\forall y\ne x_{\alpha,n+1},\;\forall z\ne x_{\alpha,n}$
		\end{enumerate}
		the induced topology by $\overrightarrow{d}$  which is generated by subbasis $$S_{\overrightarrow{d}}=\{V_{1}(x_{\alpha,n})\;|\;\alpha\in I\;,\;n\in O\}\cup \{V_{1}(x_{\alpha,n})\;|\;\alpha\in I\;,\;n\in E\}\cup\{V_{2}(p)\;|\;p\in X\}$$ $$\;\;\;\;\;\;=\{X-\{x_{\alpha,n+1}\}\;|\;n\in O\;,\;\alpha\in I\}\cup \{X-\{x_{\alpha,n-1}\}\;|\;n\in E\;,\;\alpha\in I\}\cup\{X\},$$ 
		in which E is the even natural numbers. Thus $\tau=\tau_{\overrightarrow{d}}$.
		
	\end{ex}
	\begin{ex}
		Let $\tau$ be K-topology on $\Bbb{R}$,that means the topology generated by the basis $\{(a,b)\;|\;a,b\in \Bbb{R}\}\cup \{(a,b)-\{\frac{1}{n}\}_{n\in \Bbb{N}}\;|\;a,b\in \Bbb{R}\}$, then we define R.O-metric $\overrightarrow{d}$ as follows:
		\begin{enumerate}
			\item$\forall x\in X;\;\overrightarrow{d}(x,x):=0,$\item$\forall x\notin \{\frac{1}{n}\}_{n\in {\Bbb{N}}},\;\overrightarrow{d}(x,\frac{1}{n}):=\mid{x-\frac{1}{n}}\mid+1,$\item$\; Otherwise\; \overrightarrow{d}(x,y):=|x-y|.$
		\end{enumerate}
		Then it is easy to check that
		$$S_{\overrightarrow{d}}=\{V_{r}(x)\;|\;x\notin \{\frac{1}{n}\}_{n\in {\Bbb{N}}}\;,\;0< r\le 1\}\cup \{V_{r}(x)\;|\;x\notin \{\frac{1}{n}\}_{n\in {\Bbb{N}}}\;,\;r>1\}$$$$\cup\{V_r(\frac{1}{n})\;|\;n\in\Bbb{N}\;,\;r>0\}=\{(x-r,x+r)-\{\frac{1}{n}\}_{n\in {\Bbb{N}}}\;|\;0<r\le 1\}$$$$\cup\{(x-r,x+r)\;|\;r>1\}\cup\{(\frac{1}{n}-r,\frac{1}{n}+r)\;|\;n\in\Bbb{N}\;,\;r>0\},$$
		generates topology of K-topology on $\Bbb{R}$.
	\end{ex}
	\begin{ex}
		For lower limit topology $\Bbb{R}_{l}$, which is a topology on $\Bbb{R}$ that has a basis as $\{[a,b)\;|\;a,b\in \Bbb{R}\}$, define R.O-metric $\overrightarrow{d}$ as follows:
		\begin{enumerate}
			\item$\forall a\in\Bbb{R};\;\overrightarrow{d}(a,a)=0,$\item$\overrightarrow{d}(a,b)=\begin{cases}a-b+1&b<a\\b-a&a\le b\end{cases}.$
		\end{enumerate}
		Then $$S_{\overrightarrow{d}}=\{V_r(a)\;|\;0<r\le 1\;,\;a\in\Bbb{R}\}\cup\{V_r(a)\;|\;r>1\;,\;a\in\Bbb{R}\}$$$$=\{[a,a+r)\;|\;0<r\le 1\;,\;a\in\Bbb{R}\}\cup\{(a-r,a+r)\;|\;r>1\;,\;a\in\Bbb{R}\}.$$
		Simply we can see $\tau_{\overrightarrow{d}}$ is the lower limit topology.
	\end{ex}
	Now, we give an example which shows that there can be fined R.O-metric spaces that are NOT qusi-metrizable, pseudo-metrizable and NOT quasi-pseudo-metrizable.
	\begin{ex}
		Suppose $(X,\tau_c)$ be a cofinite topological space and $Card(X)\ge Card(\Bbb{R})$. By Example \ref{13}, $(X,\tau_c)$ is R.O-metrizable. Now we prove that $(X,\tau_c)$ is not quasi-pseudo-metrizable. If it is quasi-pseudo-metrizable, then there exists a quasi-pseudo-metric $d$, such that $\tau_d=\tau_c$, thus $B=\{V_r(x)\;|\;r>0\;,\;x\in X\}$ is a basis, and for each $x\in X$ and $U$ open set containing $x$, there exists $t>0$ such that $V_t(x)\subseteq U$. Thus $B_x=\{V_r(x)\;|\;r>0\}$ is a local base at $x$, since $V_r(x)^{c}$ is finite, thus $B_x$ is at most countable. Therefore $(X,\tau_c)$ is first countable and it is a contradiction, because $Card(X)\ge Card(\Bbb{R})$ and cofinite topological spaces like $(X,\tau)$ with $Card(X)>Card(\Bbb{N})$ are not first countable $($see \cite{simmons1963introduction}$)$. Since quasi-metrizable space is quasi-pseudo-metrizable, so $(X,\tau_c)$ is not quasi-metrizable. Also if $(X,\tau_c)$ is pseudo-metrizable, then $B$ is a basis for this topology. In addition, for each $x\in X$ and $U$ open set containing $x$, there exist $t>0$ such that $V_t(x)\subseteq U$ and by the same procedure as above, it causes a contradiction.
	\end{ex}
	\begin{p}
		Suppose $(X,\overrightarrow{d})$ is a R.O-metric space and for each $x,y\in X$, define $\overrightarrow{\bar{d}}(x,y)=\frac{\overrightarrow{d}(x,y)}{\overrightarrow{d}(x,y)+1}$. Then $\tau_{\overrightarrow{d}}=\tau_{\overrightarrow{\bar{d}}}$.
	\end{p}
	\begin{proof}
		Obviously condition (1) in the definition of R.O-metric holds. Now for all $x,y\in X$, if $\overrightarrow{d}(x,y)\le \overrightarrow{d}(x,z)+\overrightarrow{d}(z,y)$, then $$\frac{1}{\overrightarrow{d}(x,z)+\overrightarrow{d}(z,y)+1}\le \frac{1}{\overrightarrow{d}(x,y)+1}$$ $$\Rightarrow 1-\frac{1}{\overrightarrow{d}(x,y)+1}\le 1-\frac{1}{\overrightarrow{d}(x,z)+\overrightarrow{d}(z,y)+1}=\frac{\overrightarrow{d}(x,z)+\overrightarrow{d}(z,y)}{\overrightarrow{d}(x,z)+\overrightarrow{d}(z,y)+1}$$ $$\Rightarrow \frac{\overrightarrow{d}(x,y)}{\overrightarrow{d}(x,y)+1}\le \frac{\overrightarrow{d}(x,z)}{\overrightarrow{d}(x,z)+1}+\frac{\overrightarrow{d}(z,y)}{\overrightarrow{d}(z,y)+1}.$$
		Therefore $\overrightarrow{\bar{d}}$ is R.O-metric on $X$. To prove $\tau_{\overrightarrow{d}}=\tau_{\overrightarrow{\bar{d}}}$, assume $V_r(x)\in  S_{\overrightarrow{d}}$ and $z\in V_r(x)$, so by definition $\overrightarrow{d}(x,z)<r$. If $\overrightarrow{d}(x,z)\ne 0$, then $$\overrightarrow{\bar{d}}(x,z)<\frac{r}{r+1}\Rightarrow z\in \bar{V}_{\frac{r}{r+1}}(x)\in S_{\overrightarrow{\bar{d}}},$$in which $\bar{V}_{\frac{r}{r+1}}(x)$ is a $r$-ball with respect to $\overrightarrow{\bar{d}}$, and if $$\overrightarrow{d}(x,z)=0,\;then\; \overrightarrow{\bar{d}}(x,z)=0<\frac{r}{r+1},$$thus $z\in\bar{V}_{\frac{r}{r+1}}(x)$ and we have $V_r(x)\subseteq \bar{V}_{\frac{r}{r+1}}(x).$ It is easy to check that $\bar{V}_{\frac{r}{r+1}}(x)\subseteq V_r(x),$ for all $x\in X$ and all non negative real numbers. Thus   $$S_{\overrightarrow{d}}\subseteq S_{\overrightarrow{\bar{d}}}\;\;\;(*).$$
		Now let $\bar{V}_r(x)\in S_{\overrightarrow{\bar{d}}}$ and $z\in \bar{V}_r(x),\;r<1$, then $\overrightarrow{\bar{d}}(x,z)<r$. If $\overrightarrow{\bar{d}}(x,z)\ne 0$, then $\overrightarrow{d}(x,z)<\frac{r}{1-r}$, and if $\overrightarrow{\bar{d}}(x,z)=0$, then $\overrightarrow{d}(x,z)=0<\frac{r}{1-r}$, thus $z\in V_{\frac{r}{1-r}}(x)\in S_{\overrightarrow{d}}$, that implies $\bar{V}_r(x)\subseteq V_{\frac{r}{1-r}}(x)$. It is easy to check that $V_{\frac{r}{1-r}}(x)\subseteq \bar{V}_r(x)$, thus we get $ S_{\overrightarrow{\bar{d}}}\subseteq S_{\overrightarrow{d}}$, and by virtue of $(*)$, $S_{\overrightarrow{\bar{d}}}= S_{\overrightarrow{d}}$, which implies $\tau_{\overrightarrow{d}}=\tau_{\overrightarrow{\bar{d}}}.$
	\end{proof}
	\begin{n}\label{1}
		It is well-known that every finite topological space $(X,\tau)$ has a subbasis $S$ such that $Card(S)\le Card(X)$, since for every point $p$ in $X$ there is the smallest open set with respect to $(\subseteq)$ containing $p$ and  the set of these open sets is a subbasis $S$ for $(X,\tau)$ , obviously $Card(S)\le Card(X)$. In the following example we show that this property does not necessarily hold for infinite topological spaces.
	\end{n}
	\begin{ex} \label{4}
		Let $Y=\Bbb{N}\times\Bbb{N}$ , $p$ be a point not in $Y$ and $X=\{p\}\cup Y$. For each function $ f:\Bbb{N}\to\Bbb{N}$, let
		$$B_f:=\{p\}\cup\{(k,\ell)\in Y\;|\;\ell\ge f(k)\}\;.$$
		Topologize $X$ by making each point of $Y$ isolated and taking $\left\{B_f:f\in{\Bbb{N}^{\Bbb{N}}}\right\}$ as a local subbasis at $p$. 
		We show that $X$ has no countable subbase.
		Let $S=\left\{B_f:f\in{\Bbb{N}^{\Bbb{N}}}\right\}\cup\{(n,m)\;|\;n,m\in\Bbb{N}\}$ and
		$$\mathscr{B}=\left\{\bigcap\mathscr{F}:\mathscr{F}\subseteq S\text{ and }\mathscr{F}\text{ is finite}\right\}\;.$$
		Thus $\mathscr{B}$ is the base generated by the subbasis $S$. $S$ is infinite and has $|S|$ finite subsets, and therefore $|\mathscr{B}|=|S|$. If $S$ is countable, $\mathscr{B}$ is also countable, and $X$ is second countable and hence first countable. But we show that there is no countable local base at $p$.
		Suppose that $\mathscr{U}=\{U_n:n\in\Bbb N\}$ is a countable family of open neighbourhoods of $p$. For each $n\in\Bbb{N}$ there are $f_{n_1},f_{n_2},\ldots,f_{n_{m_n}}\in \Bbb {N}^{\Bbb{ N}}$ such that $\bigcap_{n_1\le i\le n_{m_n}} B_{f_i}\subseteq U_n$. Define
		$$g:\Bbb N\to\Bbb N:k\mapsto 1+\max\{f_{k_i}(\ell):\ell\le k,\;1\le i\le m_k\}\;;$$
		then $B_g\nsupseteq B_{f_{n_i}}$, because it is evident that $(k,f_{k_i}(k))\in B_{f_{k_i}}$, but $(k,f_{k_i}(k))\notin B_g$ . Therefore for all $n\in \Bbb{N}$, $B_g\nsupseteq U_n$, so $\mathscr{U}$ is not a local base at $p$.
	\end{ex}
	\begin{n}
		If $X$ is infinite and $\overrightarrow{d}$ is a R.O-metric on $X$, by definition of R.O-metric and  $S_{\overrightarrow{d}}$, we can see that $Card(S_{\overrightarrow{d}})\le Card(X)$.
		Since for every $p$ in $X$, $Card(\{V_{r}(p)\;|\;r\in \Bbb{R}^{+}\})\le Card(X)$ , 
		hence $Card(S_{\overrightarrow{d}})=Card(\bigcup_{p\in X}\{V_{r}(p)\;|\;r\in \Bbb{R}^{+}\})\le Card{X}.$ This shows that the topological space in Example \ref{4} is not R.O-metrizable.
	\end{n}
	\begin{p}
		Let $(X,\tau)$ be a topological space. If it has a subbasis $S$ such that $|S|\le |X|$, and there is a function $f:X \to S$ such that  $x\in f(x)$, for every $x\in X$, then $(X,\tau)$ is R.O-metrizable.
	\end{p}
	\begin{proof}
		By the hypothesis $S=\{f(x)\;|\;x\in X\}$. For every $x,y\in X$ define $$\overrightarrow{d}(x,y):=\begin{cases} 0&y\in f(x) \\1&otherwise \end{cases}$$
		It is easy to check that $S_{\overrightarrow{d}}=S$, hence $\tau_{\overrightarrow{d}}=\tau.$
	\end{proof}
	\begin{co}
		Every finite topological space is R.O-metrizable since by Note \ref{1} we can define $f:X\to S$ such that $f(x)$ is the smallest open set containing $x$.
	\end{co}
 
\begin{cun}
	Let $(X,\tau)$ and $S$ be a subbasis of it, such that $Card\;(S)\;\le Card\;(X)$, then $(X,\tau)$ is a R.O-metrizable.
	\end{cun}

Now we mention three lemmas that will be useful for the last section.
\begin{lem}\label{2}
	Let $(X,\tau)$ be finite $T_0$-topological space, then there exists $a\in X$ such that $\{a\}\in\tau$.
\end{lem}
\begin{proof}
	Suppose that $U\in\tau$ be a minimal open set by relation $\subseteq$. If $Card\;(U)>1$, then there exists at least two points $a,b\in U$ and since $(X,\tau)$ is $T_0$, there exists $V\in\tau$ such that $a\in V,\;b\notin V$. Therefore $\emptyset\ne U\cap V\in\tau$ and $Card\;(U\cap V)<Card\;(U)$, this is a contradiction with minimality of $U$.
\end{proof}
\begin{lem}\label{15}
	Let $(X,\tau)$ be a topological space and $(\sim)$ be a relation on $X$ by :$$x\sim y\Leftrightarrow \forall U\in\tau,\;(x\in U\Leftrightarrow y\in U).$$
	Then $(\sim)$ is an equivalence relation on $X$ and $\tau^{'}=\{\;[U]\;|\;U\in \tau\}$ is a $T_0$-topology on $\mathcal{A}_{\tau}=\{\;[x]\;|\;x\in X\}$.
\end{lem}
\begin{proof}
	Obviuosely $(\sim)$ is an equivalence relation on $X$. Suppose $[z]\in [U]\cap [V]$, in which $[U],[V]\in \tau^{'}$, then iff $[z]\in [U\cap V]$. Thus $[U]\cap [V]=[U\cap V]$. Also if $[z]\in\bigcup_{\alpha\in I}[U_{\alpha}]$, then iff $[z]\in [\bigcup_{\alpha\in I}U_{\alpha}]$. Thus $\bigcup_{\alpha\in I}[U_{\alpha}]=[\bigcup_{\alpha\in I}U_{\alpha}]$. Therefore $\tau^{'}$ is a topology on $\mathcal{A}_{\tau}$. Assume $[x]\ne [y]$, thus without losing the quality, there exist $U\in \tau$ such that $x\in U$ and $y\notin U$. Thus $[x]\in [U]$ and $\bar{[y]}\subseteq [U]^c$, so $\bar{[x]}\ne \bar{[y]}$. Therefore $\tau^{'}$ is $T_0$-topology on $\mathcal{A}_{\tau}$.
\end{proof}

\begin{lem}\label{3}
	Let $(X,\tau)$ be a topological space. If $(\mathcal{A}_{\tau},\tau^{'})$ is a R.O-metrizable, then $(X,\tau)$ is R.O-metrizable.
\end{lem}

\begin{proof}
	There exist R.O-metric $\overrightarrow{d}$ such that $\tau_{\overrightarrow{d}}=\tau^{'}$. Now define for all $x,y\in X$, $$\overrightarrow{D}(x,x)=0,\;\overrightarrow{D}(x,y)=\begin{cases}0&[x]=[y]\\\overrightarrow{d}([x],[y])&[x]\ne [y]\end{cases}.$$ Obviuosly $\overrightarrow{D}$ is a metric on $X$. Assume that $[U]=\bigcup_{\alpha\in I}(\cap_{i=1}^{n_{\alpha}}U_{r_i}([x_i]))$. We clame that $U=\bigcup_{\alpha\in I}(\cap_{i=1}^{n_{\alpha}}U_{r_i}(x_i))$. So suppose that $z\in U$. Since $U=\bigcup_{[x]\in [U]}[x]$, then there exist $\alpha_0\in I$ such that $[z]\in \bigcap_{i=1}^{n_{\alpha_0}}U_{r_i}([x_i])$. Hence for all $1\le i\le n_{\alpha_0}$, $[z]\in U_{r_i}([x_i])$ which means $\overrightarrow{d}([x_i],[z])<r_i$. By definition of $\overrightarrow{D}$, we have $\overrightarrow{D}(x_i,z)<r_i$, thus for all $1\le i\le n_{\alpha_0}$, $z\in U_{r_i}(x_i)$. Therefore $U\subseteq \bigcup_{\alpha\in I}(\cap_{i=1}^{n_{\alpha}}U_{r_i}(x_i))$. Checking  $\bigcup_{\alpha\in I}(\cap_{i=1}^{n_{\alpha}}U_{r_i}(x_i))\subseteq U$ is easy. So $U=\bigcup_{\alpha\in I}(\cap_{i=1}^{n_{\alpha}}U_{r_i}(x_i))$, thus $\tau_{\overrightarrow{D}}=\tau$.
\end{proof}
	
	\section{Generalized R.O-metric space}
	\begin{defn}
		Suppose $(X,\overrightarrow{d})$ is a R.O-metric space. $(X,\overrightarrow{d},\beta)$ is called a generalized R.O-metric space if and only if there exists a collection $\beta=\{f_{\alpha}:X\to X\;|\;\alpha \in I\}$ such that $Id_X \in \beta$ and $$\forall y\in V_{r,\alpha}(x),\;\exists t>0\;,\;\exists \eta\in I \;s.t\; y\in f_{\eta}(X)\;,\;V_{t,\eta}(y)\subseteq V_{r,\alpha}(x)$$
		where $V_{r,\alpha}(x)=\{f_{\alpha}(y)\in X\;|\;d(x,f_{\alpha}(y))<r\; ,\; y\in X\; ,\; x\in f_{\alpha}(X)\}$.
	\end{defn}
	\begin{n}
		If $(X,\overrightarrow{d},\beta)$ is a generalized R.O-metric space the set $S_{\overrightarrow{d,\beta}}=\{V_{r,\alpha}(x)\;|\;x\in X\; ,\; r>0\; ,\; x\in f_{\alpha}(X)\}$ is a basis of a topology on $X$.
		Topology generated by $S_{\overrightarrow{d},\beta}$ is denoted by $\tau_{(\overrightarrow{d},\beta)}$.
	\end{n}
	\begin{defn}
		Topological space $(X,\tau)$ is called generalized R.O-metrizable iff there exists $\overrightarrow{d}$ such that $\tau_{(\overrightarrow{d},\beta)}=\tau$.
	\end{defn}
	\begin{ex} \label{5}
		Let $X$ be the set from Example \ref{4} and $\tau$ be the topology on $X$ from the same example,for every $(m,n)$ and $(s,r)$ in $\Bbb{N}^{2}$ such that $(m,n)\ne (s,r)$ define  $$\overrightarrow{d}((m,n),(m,n))=0=\overrightarrow{d}(p,(m,n))=\overrightarrow{d}(p,p)$$
		$$\overrightarrow{d}((m,n),(s,r))=1=\overrightarrow{d}((m,n),p)$$ 
		And for every $f\in \Bbb{N}^{\Bbb{N}}$ define $F_{f}:X\to X$ as follows : $$F_{f}(p)=p \;,\; F_{f}((m,n))=G((m,n))$$ Where $G:\Bbb{N}^{2} \to B_f$ is a surjective function, let $\beta=\{F_{f}\;|\;f\in \Bbb{N}^{\Bbb{N}}\} \cup \{Id_X\} $.
		Now $$V_{(\frac{1}{2} ,f)}(p)$$ $$=\{F_{f}(a)\in X\;|\;\overrightarrow{d}(p,F_{f}(a))<\frac{1}{2} \}$$ $$=\{F_{f}(a) \;|\;a\in X \}=B_f $$ and  $$V_{(\frac{1}{2} ,f)}((m,n))$$ $$=\{F_{f}(a)\in X\;|\;\overrightarrow{d}((m,n),F_{f}(a))<\frac{1}{2} \}=\{(m,n)\}$$ But if $r>1$ , $x\in X$ and $f\in \beta$ then $V_{(r,f)}(x)=X$ therefore $\tau_{(\overrightarrow{d},\beta)}=\tau$. So $(X,\tau)$ is generalize R.O-metric space which is NOT R.O-metrizable space.
	\end{ex}
	\begin{ex} \label{6}
		Let $X$ be a infinite set, $p\in X$ be fixed and $B\subset X$ which is $Card\;(B)=Card\;(B^{c})=Card\;(X)$ containing $p$. Suppose $\tau=\{A\subset X\;|\;B\subset A\;,\;A\ne B \}\cup \{\{x\}\;|\;x\in X\;,\;x\ne p\}$ then $\tau$ is a topology on $X$. Let $f_{A}:X\to X$ be surjective function such that $f_{A}(p)=p$ and $f_{A}(x)\ne p \; \forall x\ne p $ and let $\beta=\{f_{A}\;|\;A\in \tau \}\cup \{Id_X\}$. Now For distinct $x$ and $y$ in $X$ such that $x,y\ne p$ define $$\overrightarrow{d}(x,y)=0=\overrightarrow{d}(p,x)=\overrightarrow{d}(p,p)$$ $$\overrightarrow{d}(x,y)=1=\overrightarrow{d}(x,p).$$ It is easy to check that $\tau_{(\overrightarrow{d},\beta)}=\tau$ thus $(X,\tau)$ is a generalized R.O-metrizable space.
	\end{ex}

	\begin{theorem}
		Every topological space $(X,\tau)$ is generalized R.O metrizable
	\end{theorem}
	\begin{proof}
		Suppose $(X,\tau^{'})$ is an arbitrary $T_0$ topological space. Now let $(Y,\tau^{''})$ be a topological space where $Y=\{0,1\}$ and $\tau^{''}=\{\emptyset ,\;\{1\},\;Y\}$ and $\overrightarrow{d}(0,1)=0=\overrightarrow{d}(0,0)=\overrightarrow{d}(1,1),\;\overrightarrow{d}(1,0)=1$ (obviously $\tau_{\overrightarrow{d}}=\tau^{''}$). Assume that $C$ is a proper closed subset of $(X,\tau)$ , Define $f_{C}:X\to Y$ as follows :
		$f_{C}(C)=0$ and $f_{C}(X-C)=1$ 
		$f$ is obviously continuous, Let $J=\{f_{C}\;|\;X-C\in \tau\}$. $J$ is a family of continuous maps that separates points from closed sets. Define $$F:X\to Y^J\;\;\;,F(x)=(f_{\alpha}(x))_{\alpha\in J}$$ obviously $F$ is an embedding when $Y^J$ is equipped with the product topology $(F(X)\cong X)$.
		Let $ (<)$ be a well-ordering on $J$ define $$\overrightarrow{D}((x_{\alpha}),(y_{\alpha}))=\overrightarrow{d}(1,y_{\beta})$$ where $\beta$ is the smallest index in $(J,<)$  which $x_{\beta}=1$. $\overrightarrow{D}$ is a R.O metric that generates the product topology of $Y^J$ therefore $(F(X),\overrightarrow{D}|_{F(X)})$ is a R.O-metric space.We know that topology induced of $ (Y^J,\tau_{\overrightarrow{D}})$ on $F(X)$ is equal to $\{U\cap F(X)\;|\; U\in \tau_{\overrightarrow{D}}\}$ which is equal to the topology generated by $\{V_{r}(y)\cap F(X)\;|\;r>0\;,\;y\in Y^J\}$. Now let $A=\{V_{r}(y)\cap F(X)\;|\;r>0\;,\;y\in Y^J-F(X)\}$,  consider $B=\{f_{r,y}:X\to X\;|\;  y\in Y^J-F(X) \;,\; r>0\;,f_{r,y}(X)=V_{r}(y)\cap F(X)\}\cup\{Id_X\}$ we claim that $(F(X),\overrightarrow{D}|_{F(X)},B)$ is a generalized R.O-metric space, it suffices to prove that $\tau_{\overrightarrow{D}|_{F(X)},B}$ generates the induced topology of the product topology of $Y^J$ on $F(X)$. For $y\in Y^J$ and $r>0$ if $y\in F(X)$ then $V_{r}(y)\cap F(X) \in \tau_{\overrightarrow{D}|_{F(X)}}$. Also suppose $y\notin F(X)$, $r>0$ and $z\in V_{r,y}(y)=f_{r,y}(X)=V_{r}(y)\cap F(X)$, then $V_{r,y}(z)\subseteq V_{r,y}(y)$.Thus $\tau_{\overrightarrow{D}|_{F(X)},B}$ generates the induced topology of the product topology of $Y^J$ on $F(X)$. Since $F(X)\cong X$, therefore $(X,\tau^{'})$ is a generalized R.O-metrizable. By Lemma \ref{3} $(A,\tau^{'})$ is a $T_0$-topological space, thus there exist $\overrightarrow{d}$ and $\beta=\{f_{\alpha}:A\to A\}\cup \{Id_X\}$ such that $(A,\overrightarrow{d},\beta)$ is a generalized R.O-metric space. Define $$\overrightarrow{D}(x,y)=\begin{cases}0&[x]=[y]\\\overrightarrow{d}(x,y)&[x]\ne [y]\end{cases}$$ and $$F_{\alpha}:X\to X\;\;,\;\;F_{\alpha}(x)=g_{[x]}(x)$$where $g_{[x]}$ is a map between $[x]$ and $f_{\alpha}([x])$. Therefore $(X,\tau)$ is a generalized R.O-metrizable.
	\end{proof}
\noindent
At the end, it is useful to see the figure \ref{fig:16} for understanding the subject.

	
	\section*{Acknowledgements}
	 I want to thank professor Brian M. Scott for his beautiful example that helped me and thank to my dear friend Arshia Gharagozlou for helping me with editing the paper.

\newpage
\bibliographystyle{plainnat}
\bibliography{Ref}
\begin{figure}[]
	\centering
	\includegraphics[width=0.7\linewidth]{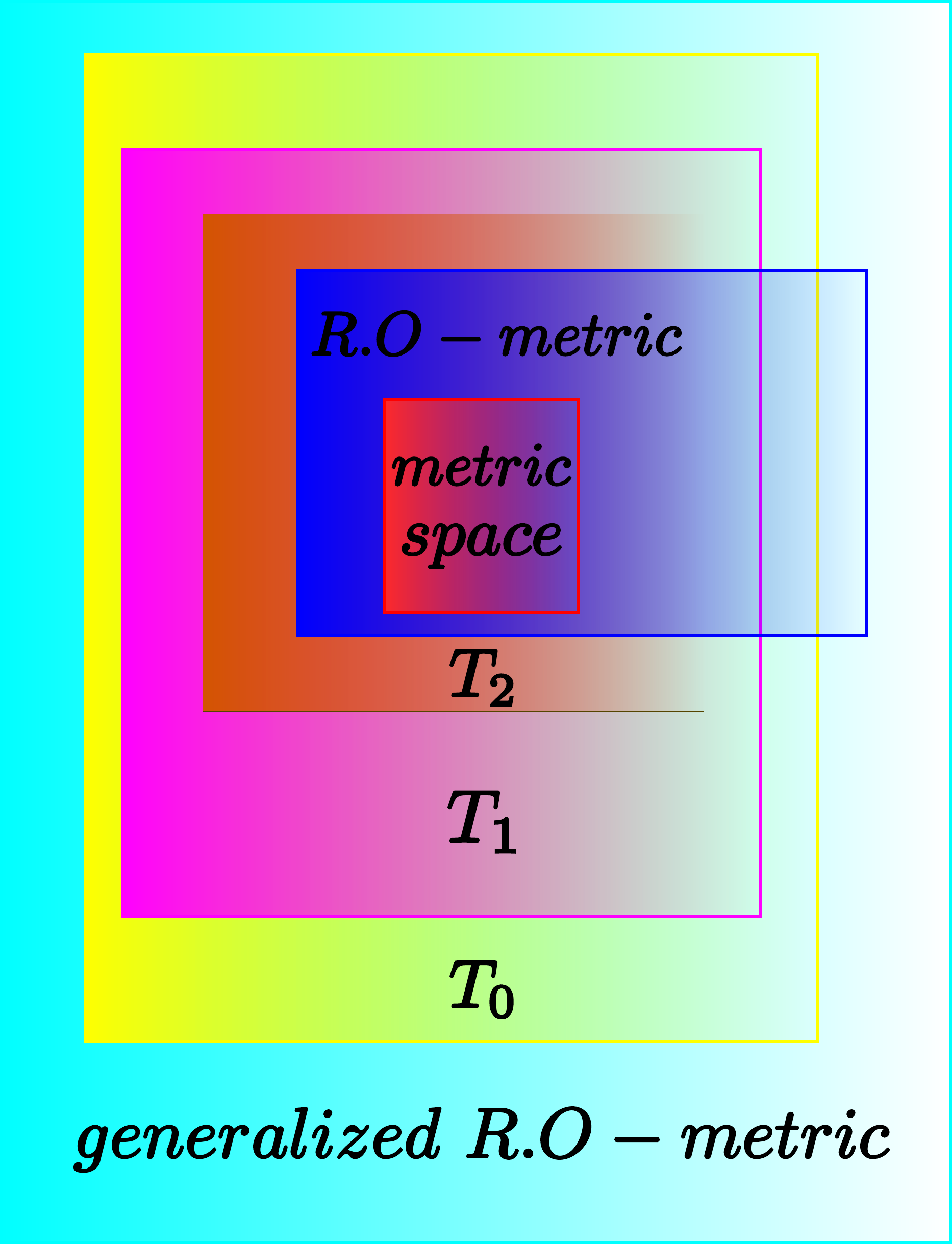}
	\caption{diagram for topological spaces}
	\label{fig:16}
\end{figure}
\end{document}